\documentclass[12pt,reqno]{amsart}

\usepackage{epsfig}
\usepackage{amsmath, amsthm, amsfonts, amssymb, mathrsfs}
\usepackage{tikz-cd}
\usepackage{graphicx}

\numberwithin{equation}{section}
\DeclareMathOperator\Hom{Hom}
\DeclareMathOperator\Ext{Ext}

\newcommand{\R}{{\mathbb R}}

\newcommand{\Z}{{\mathbb Z}}

\newcommand{\K}{{\mathbb K}}

\newtheorem{theo}{{\sc \bf Theorem}}[section]

\newtheorem{prop}[theo]{{\sc \bf Proposition}}

\begin{document}

\title{A note on Quantum Odometers}

\author[Klimek]{Slawomir Klimek}
\address{Department of Mathematical Sciences,
Indiana University-Purdue University Indianapolis,
402 N. Blackford St., Indianapolis, IN 46202, U.S.A.}
\email{sklimek@math.iupui.edu}

\author[McBride]{Matt McBride}
\address{Department of Mathematics and Statistics,
Mississippi State University,
175 President's Cir., Mississippi State, MS 39762, U.S.A.}
\email{mmcbride@math.msstate.edu}

\author[Peoples]{J. Wilson Peoples}
\address{Department of Mathematics,
Pennsylvania State University,
107 McAllister Bld., University Park, State College, PA 16802, U.S.A.}
\email{jwp5828@psu.edu}

\thanks{}

\date{\today}

\begin{abstract}
We discuss various aspects of noncommutative geometry of smooth subalgebras of Bunce-Deddens-Toeplitz Algebras.
\end{abstract}

\maketitle
\section{Introduction}
In noncommutative geometry it is often necessary to consider dense $*$-subalgebras of C$^*$-algebras, in particular, in connection with cyclic cohomology or with the study of unbounded derivations on C$^*$-algebras  \cite{Connes}. Smooth subalgebras of noncommutative spaces are also naturally present in studying spectral triples.
If C$^*$-algebras are thought of as generalizations of topological spaces, then dense subalgebras may be regarded as specifying additional structures on the underlying space, like a smooth structure.
In analogy with the algebras of smooth functions on a compact manifold, such a smooth subalgebra should be closed under holomorphic functional calculus of all elements and under smooth-functional calculus of self-adjoint elements. It should also be complete with respect to a locally convex algebra topology, see \cite{BC}.

The purpose of this note is to study smooth subalgebras $A_S^\infty$ of Bunce-Deddens-Toeplitz C$^*$-algebras $A_S$ associated to a supernatural number $S$, objects that capture their smooth structure. This work is a continuation of, and heavily relies on, our previous papers on the subject of smooth subalgebras, in particular \cite{KMP2}, \cite{KMP3} which investigated smooth structures on Bunce-Deddens algebras, the algebras of compact operators, and the Toeplitz algebra.

Bunce-Deddens algebras $B_S$ \cite{BD1}, \cite{BD2}, are crossed-product C$^*$-algebras obtained from odometers and Bunce-Deddence-Toeplitz algebras $A_S$ are their extensions by compact operators $\mathcal{K}$:  
$$0 \rightarrow \mathcal K \rightarrow A_S \rightarrow B_S \rightarrow 0.$$
Due to the topology of odometers \cite{D}, which are Cantor sets with a minimal action of a homeomorphism, the smooth subalgebras are naturally equipped with inductive limit Frechet (LF) topology.

Using a version of the Toeplitz map \cite{KMRSW2}, we build smooth subalgebras $A_S^\infty$ from Toeplitz operators with smooth symbols and from smooth compact operators. Smooth compact operators, introduced in \cite{Ph}, were studied in details in \cite{KMP3}. Smooth Bunce-Deddens algebras $B_S^\infty$, the symbols of Toeplitz operators, were introduced in \cite{KMP3}. We explicitly construct appropriate LF structures on $A_S^\infty$ and prove that those algebras are closed under holomorphic calculus so that they have the same K-Theory as their corresponding C$^*$-algebra closures, and we verify that they are closed under smooth functional calculus of self-adjoint elements.

We also focus on describing continuous derivations \cite{S} on smooth subalgebras $A_S^\infty$. In particular, using results from \cite{KMP2}, \cite{KMP3}, we classify derivations on $A_S^\infty$ and show that, up to inner derivations with compact range, they are lifts of derivations on $B_S^\infty$, the factor algebra of $A_S^\infty$ modulo the ideal $\mathcal{K}^\infty$ of smooth compact operators. Since many derivations on $B_S^\infty$ are themselves inner, the factor space of continuous inner derivations on $A_S^\infty$ modulo inner derivations turns out to be one-dimensional.
Additionally we shortly describe K-Theory and K-Homology of $A_S$.

The paper is organized as follows. Preliminary section 2 contains our notation and a short review of relevant results from \cite{KMRSW2} and \cite{KMP2}. In section 3 we review smooth compact operators and introduce and study smooth Bunce-Deddens-Toeplitz. Section 4 contains a detailed discussion of stability of $A_S^\infty$ under  both the holomorphic functional calculus, and the smooth calculus of self-adjoint elements. In sections 5 we investigate the structure and classifications of derivations. Finally, section 6 contains remarks on K-Theory and K-Homology.

\section{Preliminaries}
\subsection{Supernatural Numbers}
A {\it supernatural number} $S$ is defined as the formal product:
\begin{equation*} 
S= \prod_{p-\textrm{prime}} p^{\varepsilon_p}, \;\;\; \varepsilon_p \in\{0,1, \cdots, \infty\}\,.
\end{equation*}
We will assume $\sum \varepsilon_p = \infty$ so that $S$ an infinite supernatural number.
We define $S$-adic ring:
 $$\Z/S\Z = \prod\limits_{\substack{p-\textrm{prime}}}  \Z/{p^{\varepsilon_p}}\Z.$$ 
 Here if $S=p^\infty$ for a prime p, then $\Z/S\Z$ is equal to $\Z_p$, the ring of $p$-adic integers.
 
 If the ring $\Z/S\Z$ is equipped with the Tychonoff topology it forms a compact, Abelian topological ring with unity, though only the group structure is relevant for this paper. In addition, if $S$ is an infinite supernatural number then $\Z/S\Z$ is a Cantor set.
 
The ring $\Z / S \Z$ contains a dense copy of $\Z$ by the following indentification:
\begin{equation}\label{ZinZmodSZ}
\Z\ni k\leftrightarrow \{k\ (\textrm{mod } p^{\varepsilon_p})\}\in  \prod\limits_{\substack{p-\textrm{prime}}}  \Z/{p^{\varepsilon_p}}\Z.
\end{equation}

\subsection{Hilbert Spaces}
We use two concrete Hilbert spaces for this paper: $H=\ell^2(\Z)$ and $H_+=\ell^2(\Z_{\ge0})$ Let $\{E_l\}_{l\in\Z}$  and $\{E^+_k:k\ge0\}$ be the canonical bases for $H$ and $H_+$ respectively.
We need the following shift operator $V:H\to H$ on $H$ and the unilateral shift operator $U:H_+\to H_+$ on $H_+$: 
\begin{equation*}
VE_l = E_{l+1} \ \textrm{ and }\ UE^+_k = E^+_{k+1}.
\end{equation*}
Notice that $V$ is a unitary while $U$ is an isometry.  We have:
\begin{equation*}
[U^*,U]=P_0,
\end{equation*}
where $P_0$ is the orthogonal projection onto the one-dimensional subspace spanned by $E^+_0$.
 
For a continuous function $f\in C(\Z/S\Z)$ we define two operators $m_f:H\to H$ and $M_f:H_+\to H_+$ via formulas:
\begin{equation*}
m_fE_l = f(l)E_l\ \textrm{ and }\ M_fE^+_k = f(k)E^+_k.
\end{equation*} 
In those formulas we considered integers $k,l$ as elements of $\Z/S\Z$ using identification \eqref{ZinZmodSZ}.
Since $\Z$ is a dense subgroup of $\Z/S\Z$ we obtain immediately that
\begin{equation*}
\|m_f\|=\|M_f\|= \underset{l\in\Z}{\textrm{sup }}|f(l)| = \underset{k\in\Z_{\ge0}}{\textrm{sup }}|f(k)| =\underset{x\in \Z/S\Z}{\textrm{sup }}|f(x)| = \|f\|_\infty.
\end{equation*}
The algebras of operators generated by the $m_f$'s or by the $M_f$'s are thus isomorphic to $C(\Z/S\Z)$ and so they carry all the information about the space $\Z/S\Z$, while operators $U$ and $V$ reflect the odometer dynamics $\varphi$ on $\Z/S\Z$ given by:
\begin{equation}\label{beta_def}
\varphi(x)=x+1.
\end{equation}
The relation between those operators is:
\begin{equation}\label{Vmf_rel}
V^{-1}m_fV=m_{f\circ\varphi}.
\end{equation}
Similarly we have:
\begin{equation}\label{UMf_rel}
M_fU=UM_{f\circ\varphi}.
\end{equation}
There is also another, less obvious relation between $U$ and the $M_f$'s, namely:
\begin{equation}\label{M_with_P_zero}
M_fP_0=P_0M_f=f(0)P_0.
\end{equation}
\subsection{Algebras}
Following \cite{KMRSW2}, we define the Bunce-Deddens and Bunce-Deddens-Toeplitz algebras, $B_S$ and $A_S$ respectively, to be the following C$^*$-algebras: $B_S$ is generated by the operators $V$ and $m_f$:
\begin{equation*}
B_S=C^*\{V,m_f: f\in C(\Z/S\Z)\}
\end{equation*}
while $A_S$ is generated by the operators $U$ and $M_f$:
\begin{equation*}
A_S = C^*\{U,M_f: f\in C(\Z/S\Z)\}.
\end{equation*}
The algebra $A_S$ contains the projection $P_0$ and in fact all compact operators $\mathcal{K}$ and the quotient $A_S/\mathcal{K}$ can be naturally identified with $B_S$, see \cite{KMP2}. Let $\tau $ be the natural homomorphism $\tau:A_S\to B_S$.

The algebra $B_S$ is isomorphic with the crossed product algebra:
\begin{equation*}
B_S\cong C(\Z/S\Z)\rtimes_\varphi \Z\,.
\end{equation*} 
and is simple \cite{KMP2}. Consequently it is isomorphic the universal C$^*$-algebra with generators $v$ and $f$, where $v$ is unitary, $f\in C(\Z/S\Z)$, with relations (compare with \eqref{Vmf_rel}):
\begin{equation*}
v^{-1}fv=f\circ\varphi.
\end{equation*}
Interestingly, algebras $A_S$ can also be described in terms of generators and relations as follows.
\begin{prop} The universal C$^*$-algebra $A$ with generators $u$ and $f$, such that $u$ is an isometry, $f\in C(\Z/S\Z)$, with relations (compare with \eqref{Vmf_rel} and \eqref{M_with_P_zero}):
\begin{equation*}
fu=u\,(f\circ\varphi)\ \textrm{ and }\ fp_0=f(0)p_0,
\end{equation*}
where $[u^*,u]=p_0$, is isomorphic with $A_S$.
\end{prop}
\begin{proof} We will show that any irreducible representation of $A$ either factors through $B_S$ or is isomorphic to the defining representation of $A_S$. Since $B_S\cong A_S/\mathcal{K}$ is a factor algebra, the defining representation of $A_S$ dominates the factor representation and so, by universality, $A$ is isomorphic to $A_S$.

Consider an irreducible representation of $A$ and let $U$ represents $u$ and $M_f$ represent $f$. Notice that $P_0:=I-UU^*$ is the orthogonal projection onto the kernel of $U^*$. If that kernel is zero then $U$ is unitary and $U$, $M_f$ give a representation of $B_S$ by universality, since they satisfy the crossed-product relations.

If the kernel of $U^*$ is not zero, pick a unit vector $E^+_0$ such that $U^*E^+_0=0$. Since $U$ is an isometry, the set $\{E^+_k\}$, $k=0,1,\ldots$ is orthonormal, where $E^+_k:=U^kE^+_0$. Moreover, we have by using relations:
\begin{equation*}
M_fE^+_0=M_fP_0E^+_0=f(0)E^+_0,
\end{equation*}
and similarly:
\begin{equation*}
M_fE^+_k=M_fU^kE^+_0=U^kM_{f\circ\varphi^k}E^+_0=f(k)U^kE^+_0=f(k)E^+_k.
\end{equation*}
It follows that vectors $\{E^+_k\}$ span an invariant subspace and so, by irreducibility, $\{E^+_k\}$ is an orthonormal basis. Since $U$ is the unilateral shift in this basis, we reproduced the defining representation of $A_S$, finishing the proof.
\end{proof}

\subsection{Toeplitz Map}
Next we discuss the key relation between the two algebras $A_S$ and $B_S$.
Let $P_{\ge0}:H\to H_+$ be the following map from $H$ onto $H_+$ given by
\begin{equation*}
P_{\ge0}E_k=\left\{
\begin{aligned}
&E_k^+ &&\textrm{ if } k\ge0\\
&0 &&\textrm{ if } k<0.
\end{aligned}\right.
\end{equation*}
We also need another map $s:H_+\to H$ given by:
\begin{equation*}
sE^+_k  =E_k.
\end{equation*}
Define the map $T:B(H)\to B(H_+)$, between the spaces of bounded operators on $H$ and $H_+$, in the following way: given $b\in B(H)$ we set
\begin{equation*}
T(b) = P_{\ge0}bs.
\end{equation*}
$T$ is known as a Toeplitz map.  It has the following properties \cite{KM6}:
\begin{enumerate} 
\item $T(I_H) = I_{H_+}$. 
\item $T(bV^n) = T(b)U^n$ and $T(V^{-n}b) = (U^*)^nT(b)$ for $n\ge0$ and all $b\in B(H)$.
\item $T(bm_f) = T(b)M_f$ and $T(m_f\,b) = M_fT(b)$ for all $f\in C(\Z/S\Z)$ and all $b\in B(H)$
\item $T(b^*) = T(b)^*$ for all $b\in B(H)$.
\end{enumerate}
Consequently, it follows that $T$ is a $*$-preserving map from $B_S$ to $A_S$.
If $\tau $ is the natural homomorphism from $A_S$ to $B_S$ then we have 
$$\tau T(b) =b$$ for all $b\in B_S$. It follows that for any $a$ in $A_S$ there is a compact operator $c$ such that we have a decomposition:
\begin{equation}\label{ASdecomp}
a=T(b)+c,
\end{equation}
where $b=\tau(a)\in B_S$.
One can verify that if $b$ is an element in $B_S$ then $T(b)$ is compact if and only if $b=0$. This implies the uniqueness 
of the above decomposition \eqref{ASdecomp}.

\subsection{Fourier Series}
There are natural one-parameter groups of automorphisms of $B_S$ and $A_S$ respectively.  They are given by the formulas:
\begin{equation*}
\rho^{\mathbb L}_\theta(b) = e^{2\pi i\theta\mathbb{L}}be^{-2\pi i\theta\mathbb{L}}\ \textrm{ for }b\in B_S\ \textrm{ and }\ 
\rho^\K_\theta(a) = e^{2\pi i\theta\K}ae^{-2\pi i\theta\K}\ \textrm{ for }a\in A_S,
\end{equation*}
where $\theta\in\R/\Z$. Here we using the following diagonal label operators on $H$ and $H_+$ respectively:  
$$\mathbb{L} E_l = lE_l\ \textrm{ and }\ \K E^+_k = kE^+_k.$$ 
We have the following relations:
\begin{equation*}
\rho^{\mathbb L}_\theta(V)=e^{2\pi i\theta}V\textrm{ and } \rho^{\mathbb L}_\theta(m_f)=m_f\,.
\end{equation*}
Automorphisms $\rho^\K_\theta$ satisfy analogous relations and the extra relation on $U^*$, namely
\begin{equation*}
\rho^\K_\theta(U^*)=e^{-2\pi i\theta}U^*\,.
\end{equation*}

Define $E:B_S\to C^*\{m_f:f\in C(\Z/S\Z)\}\cong C(\Z/S\Z)$ via
\begin{equation*}
E(b) =\int_0^1\rho^{\mathbb L}_\theta(b)\,d\theta\,.
\end{equation*}
It's easily checked that $E$ is an expectation on $B_S$.  For a $b\in B_S$ we define the {\it $n$-th Fourier coefficient} $b_n$ by the following:
\begin{equation*}
b_n=e(V^{-n}b)= \int_0^1\rho_\theta(V^{-n}b)\,d\theta = \int_0^1e^{-2\pi in\theta}V^{-n}\rho_\theta(b)\,d\theta.
\end{equation*}
From this definition, it's clear that $b_n\in C^*\{M_f:f\in C(\Z/S\Z)\}$ so we can write $b_n = m_{f_n}$ for some $f_n\in C(\Z/S\Z)$.  We define an expectation, $E$ on $A_S$, in a similar fashion:
\begin{equation*}
E:A_S\to C^*\{M_f:f\in C(\Z/S\Z)\}\cong C(\Z/S\Z)\,.
\end{equation*}
For an $a\in A_S$, its $n$-th Fourier coefficient $a_n$ is also defined similarly and also $a_n=M_{f_n}$ for some $f_n\in C(\Z/S\Z)$. 
Additionally, notice that we have the following relation with the Toeplitz map:
\begin{equation*}
(T(b))_n = T(b_n)\ \textrm{ for all }n\,.
\end{equation*}

\section{Smooth Subalgebras}
\subsection{Smooth Compact Operators}
We begin by reviewing properties of smooth compact operators from \cite{KMP3}.
Let $\mathcal{K}$ be the algebra of compact operators on $H_+$.  The orthonormal basis $\{E^+_k\}_{k\ge0}$ of $H_+$ determines a system of units $\{P_{ks}\}_{k,s\ge0}$ in $\mathcal{K}$ that satisfy the following relations:
\begin{equation*}
P_{ks}^* = P_{sk} \quad\textrm{and}\quad P_{ks}P_{rt} = \delta_{sr}P_{kt}\,,
\end{equation*}
where $\delta_{sr}=1$ for $s=r$ and is equal to zero otherwise.  The set of {\it smooth compact operators} with respect to $\{E^+_k\}$ is the set of operators of the form
\begin{equation*}
c = \sum_{k,s\ge0}c_{ks}P_{ks}\,,
\end{equation*}
so that the coefficients $\{c_{ks}\}_{k,s\ge0}$ are rapidly decaying (RD).  We denote the set of smooth compact operators by $\mathcal{K}^\infty$.

We now introduce norms on $\mathcal{K}^\infty$. They are constructed using the following useful derivation on $\mathcal{K}^\infty$:
\begin{equation*}
d_\K(c) = [\K,c]\,.
\end{equation*}
Clearly $d_\K$ is linear and satisfies the the Leibniz rule as $d_\K$ is a commutator. 
We define $\|\cdot\|_{M,N}$ norms on $\mathcal{K}^\infty$ by the following formulas:
\begin{equation*}
\|c\|_{M,N} = \sum_{j=0}^M\begin{pmatrix} M \\ j \end{pmatrix}\|d_\K^j(c)(I+\K)^N\|\,,
\end{equation*}
with $\delta_\K^0(c):=c$. 
The following proposition from \cite{KMP3} summarizes the basic properties of $\|\cdot\|_{M,N}$ norms.
\begin{prop}\label{MN-norm_basics}
Let $a$ and $b$ be bounded operators in $H$, then
\begin{enumerate}
\item $a\in\mathcal{K}^\infty$ if and only if $\|a\|_{M,N}<\infty$ for all nonnegative integers $M$ and $N$.
\item $\|a\|_{M+1,N} = \|a\|_{M,N} + \|d_\K(a)\|_{M,N}$.
\item $\|a\|_{M,N} \le \|a\|_{M,N+1}$.
\item $\|ab\|_{M,N} \le \|a\|_{M,0}\|b\|_{M,N} \le \|a\|_{M,N}\|b\|_{M,N}$.
\item $\|d_\K(a)\|_{M,N}\le \|a\|_{M+1,N}$.
\item $\|a^*\|_{M,N}\le \|a\|_{M+N,N}$.
\item $\mathcal{K}^\infty$ is a complete topological vector space.
\end{enumerate}
\end{prop}
This proposition implies that $\mathcal{K}^\infty$ is a Fr\'{e}chet $*$-algebra with respect to the norms, $\|\cdot\|_{M,N}$.

\subsection{Smooth Bunce-Deddens Algebras}
Next we review  smooth Bunce-Deddens algebras $B_S^\infty$ from \cite{KMP2}.
We need the following terminology.  We say a family of locally constant functions on $\Z/S\Z$ is {\it Uniformly Locally Constant}, ULC, if there exists a divisor $l$ of $S$ such that for every $f$ in the family we have
\begin{equation*}
f(x+l) = f(x)
\end{equation*}
for all $x\in\Z/S\Z$.   

We define the space of smooth elements of the Bunce-Deddens algebra, $B_S^\infty$, to be the space of elements in $B_S$ whose Fourier coefficients are ULC and whose norms are RD.  Using Fourier series those conditions can be written as:
\begin{equation*}
B_S^\infty=\left\{b= \sum_{n\in\Z}V^nm_{f_n} : \{\|f_n\|\}\textrm{ is }RD, \textrm{ there is an }l|S,\,\,V^lbV^{-l}=b\right\}\,.
\end{equation*}
It's immediate that $B_S^\infty$ is indeed a nonempty subset of $B_S$ and it was proved in \cite{KMP2} that $B_S^\infty$ is a $*$-subalgebra of $B_S$. 

Let $\delta_{\mathbb L}:B^\infty_S \to B^\infty_S$ be given by
\begin{equation*}
\delta_{\mathbb L}(b) = [\mathbb L,b]
\end{equation*}
This derivation is very fundamental below. We have the following simple relations:
\begin{equation*}
\delta_{\mathbb L}(v^n) =nV^n\ \textrm{ and }\ \delta_{\mathbb L}(m_f) =0.
\end{equation*}
This derivative is in particular used to define the following norms on $B_S^\infty$ that capture the RD property of the Fourier coefficients of elements of $B_S^\infty$.
They are defined by:
\begin{equation*}
\|b\|_P = \sum_{j=0}^P\
\begin{pmatrix}
P \\ j
\end{pmatrix}\|\delta_{\mathbb L}^{j}(b)\|\,.
\end{equation*}
The following proposition from \cite{KMP2} states the basic properties the $P$-norms. 

\begin{prop}\label{P-norm_basics}
Let $b_1$ and $b_2$ be in $B_S^\infty$, then
\begin{enumerate}
\item $\|b_1\|_{P+1} = \|b_1\|_P + \|\delta_{\mathbb L}(b_1)\|_P$ with $\|b_1\|_0:=\|b_1\|$.
\item $\|b_1b_2\|_P \le \|b_1\|_P\|b_2\|_P$.
\item $\|\delta_{\mathbb L}(b_1)\|_P\le \|b_1\|_{P+1}$.
\end{enumerate}
\end{prop}
It follows that we have the following useful way to describe elements in $B_S^\infty$:
\begin{equation*}
B_S^\infty = \{b\in B_S: \|b\|_M<\infty\,,\textrm{ for every }M,\textrm{ there is an }l|S,\,\,V^lbV^{-l}=b\}\,.
\end{equation*}

\subsection{Smooth Bunce-Deddens-Toeplitz Algebras}
Finally, following similar considerations for the Toeplitz algebra in \cite{KMP3}, we define the smooth Bunce-Deddens-Toeplitz algebra $A_S^\infty$ by
\begin{equation*}
A_S^\infty = \{a = T(b) + c : b\in B_S^\infty,\ c\in\mathcal{K}^\infty\}\subseteq A_S\,.
\end{equation*}
Much like with the short exact sequence for $A_S$ and $B_S$, these smooth subalgebras have the following related short exact sequence:
\begin{equation*}
0\longrightarrow \mathcal{K}^\infty \longrightarrow A_S^\infty \longrightarrow B_S^\infty\longrightarrow 0\,.
\end{equation*}
Thus, we can view the topology on $A_S^\infty$, as a vector space, in the usual way:
\begin{equation*}
A_S^\infty \cong B_S^\infty\oplus \mathcal{K}^\infty\,.
\end{equation*}
This gives $A_S$ its LF topology.

The Toeplitz map $T:B_S\to A_S$ can naturally be restricted to $B_S^\infty$ and considered as a map $T:B_S^\infty\to A_S^\infty$.  In addition, the homomorphism $\tau $ can be restricted to $A_S^\infty$ and we have a homomorphism $\tau :A_S^\infty\to B_S^\infty$.  

It is easy to verify on generators that we have
\begin{equation*}
d_\K(T(b)) = T(\delta_{\mathbb L}(b)).
\end{equation*}
As a consequence of continuity of $T$ this formula is true for all $b\in B_S^\infty$.

It remains to verify that $A_S^\infty$ is indeed a subalgebra of $A_S$. This follows from the following two propositions.

\begin{prop}\label{smooth_bdt_idealish}
Let $b$ be in $B_S^\infty$ and $c$ be in $\mathcal{K}^\infty$.  Then $T(b)c$ and $cT(b)$ are in $\mathcal{K}^\infty$.
\end{prop}
\begin{proof}
Because $T(b^*)=T(b)^*$, we only need to prove $T(b)c$ is in $\mathcal{K}^\infty$.  Proceeding as in \cite{KMP3} we prove by induction on $M$ that we have the following estimate:
\begin{equation}\label{leftT(f)}
\|T(b)c\|_{M,N}\leq \|b\|_{M}\|c\|_{M,N}\,.
\end{equation}
The $M=0$ case is immediate from the definition of the norms. The inductive step is:
\begin{equation*}
\begin{aligned}
&\|T(b)c\|_{M+1,N}=\|T(b)c\|_{M,N}+\|d_\K(T(b))c+T(b)d_\K(c)\|_{M,N}\leq \\
&\leq \left(\|b\|_{M}+\|\delta_{\mathbb{L}}(b)\|_{M}\right)(\|c\|_{M,N}+\|d_\K(c)\|_{M,N})=\|b\|_{M+1}\|c\|_{M+1,N}\,.
\end{aligned}
\end{equation*}
Notice also that, again proceeding as in \cite{KMP3}, we can obtain the following inequality:
\begin{equation}\label{rightT(f)}
\|cT(b)\|_{M,N}\leq \|b\|_{M+N}\|c\|_{M,N}\,.
\end{equation}
\end{proof}

\begin{prop}\label{smooth_bd_T_almost_hom}
Let $b_1$ and $b_2$ be smooth Bunce-Deddens elements, then the following expression is a smooth compact element:
\begin{equation*}
T(b_1)T(b_2) - T(b_1b_2)\,.
\end{equation*}
\end{prop}
\begin{proof}
We follow  \cite{KMP3}. Let $b_1$ and $b_2$ be in $B_S^\infty$ with the following decompositions:
\begin{equation*}
b_1=b_1^+ + b_1^-=\sum_{n\ge0}V^nm_{f_n} + \sum_{n<0}m_{f_n}V^n\quad\textrm{and}\quad b_2=b_2^+ + b_2^-=\sum_{n\ge0}V^nm_{g_n} + \sum_{n<0}m_{g_n}V^n
\end{equation*}
where $\{\|f_n\|\}$ and $\{\|g_n\|\}$ are RD sequences and $\{f_n\}$ and $\{g_n\}$ are ULC.  Since $T$ is linear we only need to study the following differences:
\begin{equation*}
\begin{aligned}
&T(b_1^+)T(b_2^+)-T(b_1^+b_2^+),  &&T(b_1^-)T(b_2^-)-T(b_1^-b_2^-)\\
&T(b_1^-)T(b_2^+)-T(b_1^-b_2^+), &&T(b_1^+)T(b_2^-)-T(b_1^-b_2^+)\,.
\end{aligned}
\end{equation*}
First consider the following:
\begin{equation*}
\begin{aligned}
T(b_1^+)T(b_2^+) - T(b_1^+b_2^+) &= \sum_{m,n\ge0}U^nM_{f_n}U^mM_{g_m} - \sum_{m,n\ge0}T\left(V^nm{f_n}V^mm_{g_m}\right) \\
&=\sum_{m,n\ge0}U^{n+m}M_{f_n\circ\varphi^m}M_{g_m} - \sum_{m,n\ge0}T\left(V^{n+m}m_{f_n\circ\varphi^m}m_{g_m}\right) \\
&=\sum_{m,n\ge0}U^{n+m}M_{f_n\circ\varphi^m}M_{g_m} - \sum_{m,n\ge0}T\left(V^{n+m}\right)M_{f_n\circ\varphi^m}M_{g_m}.
\end{aligned}
\end{equation*}
Since $T(V^{n+m}) = U^{n+m}$, so the above is zero.  A similar argument can be made for $T(b_1^-)T(b_2^-)-T(b_1^-b_2^-)$.  For the next difference we have
\begin{equation*}
T(b_1^-)T(b_2^+)-T(b_1^-b_2^+) = \sum_{m\ge0,n<0}M_{f_n}(U^*)^{-n}U^mM_{g_m}-\sum_{m\ge0,n<0}M_{f_n}T(V^nV^m)M_{g_m}.
\end{equation*}
However, since $T(V^{n+m}) = (U^*)^{-n}U^m$ since $n<0$, this difference is also zero.  Finally, for the last difference, we have
\begin{equation*}
\begin{aligned}
C:=T(b_1^+)T(b_2^-)-T(b_1^+b_2^-) &= T(b_1^+)\sum_{m<0}M_{g_m}(U^*)^{-m}-\sum_{m<0}T(b_1^+m_{g_m}V^m) \\
&=\sum_{m<0}\left(T(b_1^+m_{g_m})(U^*)^{-m}-T(b_1^+m_{g_m}V^m)\right)\\
&=-\sum_{m<0}T(b_1^+m_{g_m}V^m)P_{<-m}
\end{aligned}
\end{equation*}
where we used the following formula for $m<0$:
\begin{equation*}
U^{-m}(U^*)^{-m} - I = -P_{<-m}\,.
\end{equation*}
Clearly, $C$ is compact but we still need to prove it's smooth compact. To this end, we prove the $M,N$-norms of $C$ are finite.  A straightforward calculation gives:
\begin{equation*}
d_\K^{j}(C) = -\sum_{m<0}d_\K^{j}\left(T(b_1^+m_{g_m}V^m)P_{<-m}\right) = -\sum_{m<0}T\left(d_\K^{j}(b_1^+V^m)\right)P_{<-m}
\end{equation*}
Next we estimate norms of $C$ using $\|P_{<-m}\|_{0,N} =|m|^N$ to obtain:
\begin{equation*}
\begin{aligned}
\|d_\K^{j}(C)\|_{0,N}&\le \sum_{m<0}\sum_{l=0}^j\begin{pmatrix} j \\ l\end{pmatrix} |m|^{j-l+N}\|d_\K^{l}(b_1^+)\|\|g_m\| \\
&\le \sum_{m<0}(1+|m|)^{N+j}\left(\sum_{l=0}^j\begin{pmatrix}j\\l\end{pmatrix}\|d_\K^{l}(b_1^+)\|\right)\|g_m\|\\
&=\sum_{m<0}\|b_1^+\|_j(1+|m|)^{N+j}\|g_m\| \leq \textrm{const} \|b_1^+\|_j\|b_2^-\|_{N+j+2}\,.
\end{aligned}
\end{equation*}
Consequently, since $b_1$ and $b_2$ are in $B_S^\infty$ we get $||C||_{M,N}<\infty$.
This shows $T(b_1)T(b_2)-T(b_1b_2)$ is smooth compact. A more careful analysis following \cite{KMP3} yields the following estimate:
\begin{equation}
\label{TProd_estimate}
\|T(b_1)T(b_2)-T(b_1b_2)\|_{M,N} \le \textrm{const}\|b_1\|_j\|b_2\|_{N+j+2}\\,.
\end{equation}

\end{proof}

\section{Stability of Smooth Bunce-Deddens-Toeplitz Algebra}
The purpose of this section is to establish stability of $A_S^\infty$ under both the holomorphic functional calculus, and the smooth calculus of self-adjoint elements. It is well known that showing the former automatically implies that the $K$-Theories of $A_S^\infty$ and $A_S$ coincide \cite{Bo}.

\begin{prop}
The smooth Bunce-Deddens-Toeplitz algebra $A_S^\infty$ is closed under the holomorphic functional calculus.
\end{prop}

\begin{proof}
Since $A_S^\infty$ is a complete locally convex topological vector space, it is enough to check that if $a\in A_S^\infty$ and invertible in $A_S$, then $a^{-1}\in A_S^\infty$.  Consequently, the Cauchy integral representation finishes the proof.  To this end, let $a\in A_S^\infty$ and thus $a=T(b) +c$ with $b\in B_S^\infty$ and $c\in\mathcal{K}^\infty$ and suppose $a$ is invertible in $A_S$.  Since $\tau $ is a homomorphism, $\tau (a)=b$ is invertible in $B_S^\infty$. It is proved in \cite{KMP2} that if $b\in B_S^\infty$ and invertible, then $b^{-1}\in B_S^\infty$.   Since $\mathcal{K}$ is an ideal of $A_S$ and $\tau T$ is the identity map, it follows that
\begin{equation*}
a^{-1} = T(b^{-1}) + c'
\end{equation*}
for some $c'\in\mathcal{K}$.  The proof will be complete if we can show that $c'\in\mathcal{K}^\infty$.  Notice that
\begin{equation*}
c' = a^{-1} - T(b^{-1}) = a^{-1}(I-aT(b^{-1}) = a^{-1}(I - T(b)T(b^{-1}) + cT(b^{-1}))\,.
\end{equation*}
From Propositions \ref{smooth_bdt_idealish} and \ref{smooth_bd_T_almost_hom}, we have that both $I-T(b)T(b^{-1})$ and $cT(b^{-1})$ are in $\mathcal{K}^\infty$.  Consequently, there is a $\tilde{c}\in\mathcal{K}^\infty$ such that $c'=a^{-1}\tilde{c}$.  It follows from the properties of norms on $\mathcal{K}^\infty$ that
\begin{equation}\label{inverse_est}
\|c'\|_{0,N}\le\|a^{-1}\|\|\tilde{c}\|_{0,N}<\infty\,.
\end{equation}
Computing $\delta_\K$ on $c$ we have
\begin{equation*}
\delta_\K(c') = \delta_\K(a^{-1})\tilde{c}) = -a^{-1}\delta_\K(a)a^{-1}\tilde{c} + a^{-1}\delta_\K(\tilde{c})\,.
\end{equation*}
Similarly to the proof of Proposition \ref{smooth_bdt_idealish}, we have, inductively, for any $j$ that
\begin{equation*}
\delta_\K^j(b) = \sum_ia_ib_i \quad\textrm{finite sum,}
\end{equation*}
with $a_i$ bounded and $b_i$ are smooth compact.  Using this and the estimate in equation (\ref{inverse_est}), we see that $\|c'\|_{M,N}$ is finite for all $M$ and $N$.  Thus $c'\in\mathcal{K}^\infty$, completing the proof.
\end{proof}
To prove closure under the calculus of self-adjoint elements, the approach used in \cite{KMP2} works in this setting as well. Hence, we need results regarding the growth of exponentials of elements of $B_S^\infty$ and $\mathcal{K}^\infty$. For $\mathcal{K}^\infty$, the exact result needed was proved in \cite{KMP2}. We state it here for convenience.

\begin{prop}\label{exp_of_c}
Suppose that $c\in\mathcal{K}^\infty$ is a self-adjoint smooth compact operator. Then we have an estimate:
\begin{equation*}
\|e^{ic}\|_{M,0}\leq \prod_{j=1}^M (1+\|c\|_{j,0})^{2^{M-j}}\,.
\end{equation*}
\end{prop}
The second result needed is a minor adaptation of Proposition 3.4 in \cite{KMP2}. 
\begin{prop}\label{exp_of_b}
If $b\in B^\infty_S$ is self-adjoint, then we have an estimate:
\begin{equation*}
\|e^{ib}\|_{M} \leq \prod_{j=1}^M (1+\|b\|_{j})^{2^{M-j}}\,.
\end{equation*}
\end{prop}
\begin{proof} 
For $M=0$, notice that $\|e^{ib}\|_{0} = 1$. We continue by induction, utilizing part (1) of Proposition \ref{P-norm_basics}:
$$
\|e^{ib} \|_{M+1} = \|e^{ib}\|_M + \|\delta_{\mathbb{L}}(e^{ib}) \|_M.
$$
Using that 
\begin{equation*}
\delta_\mathbb{L}(e^{ib})=i\int_0^1e^{i(1-t)b}\delta_\mathbb{L}(b)e^{itb}\,dt\,,
\end{equation*}
we have the following estimate for the inductive step:
\begin{equation*}
\begin{aligned}
&\|e^{ib}\|_{M+1}\leq \|e^{ib}\|_{M}+i\int_0^1\|e^{i(1-t)b}\|_{M}\|\delta_\mathbb{L}(b)\|_{M}\|e^{itb}\|_{M}\,dt\leq\\
&\leq \prod_{j=1}^M (1+\|b\|_{j})^{2^{M-j}} +\left[\prod_{j=1}^M (1+\|b\|_{j})^{2^{M-j}}\right]^2\|\delta_\mathbb{L}(b)\|_{M}\,.
\end{aligned}
\end{equation*}
Since $\|\delta_\mathbb{L}(b)\|_{M}\leq\|b\|_{M+1}$, we have:
\begin{equation*}
\begin{aligned}
&\|e^{ib}\|_{M+1}\leq \prod_{j=1}^M (1+\|b\|_{j})^{2^{M-j}}(1+\prod_{j=1}^M (1+\|b\|_{j})^{2^{M-j}}\|b\|_{M+1})\leq\\
&\leq \prod_{j=1}^M (1+\|b\|_{j})^{2^{M-j}}\prod_{j=1}^M (1+\|b\|_{j})^{2^{M-j}}(1+\|b\|_{M+1})=\prod_{j=1}^{M+1} (1+\|b\|_{j})^{2^{M+1-j}}\,.
\end{aligned}
\end{equation*}
This establishes the inductive step and finishes the proof.
\end{proof}

\begin{theo}
The smooth Bunce-Deddens-Toeplitz algebra $A_S^\infty$ is closed under the smooth functional calculus of self-adjoint elements.\end{theo}
\begin{proof}  We need to prove that, given a self-adjoint element $a$ of $A_S^\infty$ and a smooth function $f(x)$ defined on an open neighborhood of the spectrum $\sigma(a)$ of $a$ we have $f(a)$ is in $A_S^\infty$.
It is without loss of generality to assume that $f(x)$ is smooth on $\R$ and is $L$-periodic: $f(x+L)=f(x)$ for some $L$.. Then $f(x)$ admits a Fourier series representation with rapid decay coefficients $\{f_n\}$, and hence
\begin{equation*}
f(a)=\sum_{n\in\Z}f_ne^{2\pi ina/L}
\end{equation*}
for a self-adjoint $a=T(b)+c\in A_S^\infty$. 
Thus, it remains to establish at most polynomial growth in $n$ of norms $\|e^{2\pi ina/L}\|_{M,N}$.

Notice that ${\tau }\left(e^{2\pi ina/L}\right)$ in $B_S^\infty$ is $e^{2\pi in b/L}$, which indeed grows at most polynomially in $n$, by Proposition \ref{exp_of_b}. Thus, we only need to show that the $\|\cdot\|_{M,N}$ of the difference
\begin{equation*}
e^{2\pi in(T(b)+c)/L}-T\left(e^{2\pi in b/L}\right)\in\mathcal{K}^\infty
\end{equation*}
are at most polynomially growing in $n$. 

To analyze the above, we use a version of the Duhamel's formula:
\begin{equation*}
\begin{aligned}
&e^{i(T(b)+c)}-T\left(e^{ib}\right)=\int_0^1\frac{d}{dt}\left( e^{it(T(b)+c)}T\left(e^{i(1-t)b}\right)\right)dt=\\
&=\int_0^1 e^{it(T(b)+c)}c\,T\left(e^{i(1-t)b}\right)dt +\int_0^1 e^{it(T(b)+c)}\left[T(b)T\left(e^{i(1-t)b}\right)-T\left(be^{i(1-t)b}\right)\right]dt\,.
\end{aligned}
\end{equation*}
Employing Proposition \ref{MN-norm_basics} we can estimate the norms as follows:
\begin{equation*}
\begin{aligned}
&\|e^{i(T(b)+c)}-T\left(e^{ib}\right)\|_{M,N}\leq \int_0^1 \|e^{it(T(b)+c)}\|_{M,0}\|c\,T\left(e^{i(1-t)b}\right)\|_{M,N}\,dt+\\
& +\int_0^1 \|e^{it(T(b)+c)}\|_{M,0}\|T(b)T\left(e^{i(1-t)b}\right)-T\left(be^{i(1-t)b}\right)\|_{M,N}\,dt\,.
\end{aligned}
\end{equation*}
All terms above can now be estimated using \eqref{rightT(f)},  as well as Propositions \ref{exp_of_c} and \ref{exp_of_b}.
We obtain the following bounds:
\begin{equation*}
\begin{aligned}
&\|e^{i(T(b)+c)}-T\left(e^{ib}\right)\|_{M,N}\leq \prod_{j=1}^M (1+\|b\|_{j}+\|c\|_{j,0})^{2^{M-j}}\,\|c\|_{M,N}\prod_{j=1}^{M+N} (1+\|b\|_{j})^{2^{M+N-j}}+\\
& + \textrm{const} \prod_{j=1}^M (1+\|b\|_{j}+\|c\|_{j,0})^{2^{M-j}}\, \|b\|_{M} \prod_{j=1}^{M+N+2} (1+\|b\|_{j})^{2^{M+N+2-j}}\,.
\end{aligned}
\end{equation*}
Clearly those estimates establish the desired at most polynomial growth, finishing the proof.
\end{proof}
\section{Classification of Derivations}
We begin with recalling the basic concepts from \cite{KMRSW2}.  Let $A$ be a complete locally compact topological algebra and let $d:A\to A$ be continuous derivation on $A$. Suppose that there is a continuous one-parameter family of automorphisms $\rho_\theta:A\to A$ of $A$, $\theta\in\R/\Z$.

Given $n\in\Z$, a continuous derivation $d:A\to A$ is said to be a {\it $n$-covariant derivation} if the relation 
\begin{equation*}
\rho_\theta^{-1}d\rho_\theta(a)= e^{-2\pi in\theta} d(a)
\end{equation*}
holds for all $\theta$.  When $n=0$ we say the derivation is invariant.  In this definition $A$ could be any of the following algebras: $A_S^\infty$, $B_S^\infty$, or $\mathcal{K}^\infty$ and with the appropriate one-parameter family of automorphisms $\rho_\theta^\K$ or $\rho_\theta^{\mathbb L}$.  With this definition, we point out that $\delta_{\mathbb{L}}:B_S^\infty\to B_S^\infty$ is an invariant continuous derivation as is $d_\K:A_S^\infty\to A_S^\infty$ and $d_\K:\mathcal{K}^\infty\to\mathcal{K}^\infty$.

If $d$ is a continuous derivation on $A$, the {\it $n$-th Fourier component} of $d$ is defined as: 
\begin{equation*}
d_n(a)= \int_0^1 e^{2\pi in\theta} \rho_\theta^{-1}d\rho_\theta(a)\, d\theta\,.
\end{equation*}
We have the following simple observation \cite{KMRSW2}.
\begin{prop}\label{d_n_cov}
With the above notation the $n$-th Fourier component $d_n:A_S^\infty\to A_S^\infty$ is a continuous $n$-covariant derivation.
\end{prop}

To classify continuous derivations on $A_S^\infty$  we follow the strategy from \cite{KMRSW2}. We use the classification of derivations on $B_S^\infty$ from \cite{KMP2} and show how to lift derivations from $B_S$ to $A_S$. We handle the remaining derivations, those with range in $\mathcal{K}^\infty$, by using the Fourier decomposition components. This is the heart of the argument and will be described next.

Let $\mathcal{A}_S\subseteq A_S^\infty$ be the subspace of $A_S^\infty$ consisting of elements $a=T(b)+c$ such that $b$ has only finitely many non-zero Fourier components and $c$ has only finitely many non-zero matrix coefficients (in the standard basis).
 It was observed in \cite{KMRSW2}  that $\mathcal{A}_S$ is a dense subalgebra of $A_S$.  In turn, we note that it is also a dense subalgebra of $A_S^\infty$.

\begin{theo}\label{cont_der_ran_com_inner}
If  $d:A_S^\infty\to\mathcal{K}^\infty$ is a continuous derivation, then there is $c\in \mathcal{K}^\infty$ such that $d(a)=[c,a]$ for every $a\in A_S^\infty$. In particular, $d$ is an inner derivation.
\end{theo}
\begin{proof}
Let $d:A_S^\infty\to\mathcal{K}^\infty$ be a continuous derivation.  Let $d_n$ be the nth-Fourier component of $d$.  From Proposition \ref{d_n_cov}, $d_n$ are $n$-covariant derivations and $d_n:A_S^\infty\to\mathcal{K}^\infty$.  We only consider the case $n\ge0$ as $n<0$ can be treated similarly. All  $n$-covariant derivations $d_n:\mathcal{A}_S\to A_S$ were classified in  \cite{KMRSW2} .  Thus, we know there exists a sequence, $\{\beta_n(k)\}$, possibly unbounded in $k$, such that
\begin{equation}\label{d_n_formula}
d_n(a) = [U^n\beta_n(\K),a]
\end{equation}
for any $a\in \mathcal{A}_S$.  We are requiring here the range of $d$ to belong to $\mathcal{K}^\infty$, which places restrictions on $\{\beta_n(k)\}$. 

Let $\chi$ be a character on $\Z/S\Z$ and since $d_n(a)\in\mathcal{K}^\infty$ for any $a\in A_S^\infty$ we have
\begin{equation*}
\left\{
\begin{aligned}
d_n(U) &= U^{n+1}(\beta_n(\K+I)-\beta_n(\K)):=U^{n+1}\alpha_n(\K)\in\mathcal{K}^\infty &&\textrm{for }n\ge0\\
d_n(M_\chi) &= U^n\beta_n(\K)(1-\chi(n))\in\mathcal{K}^\infty &&\textrm{for }n\geq0\,.
\end{aligned}\right.
\end{equation*}
Since for each $n>0$ we can choose $\chi$ such that $\chi(n)\neq 1$, and thus we have $\{\alpha_n(k)\}$ and $\{\beta_n(k)\}$ are RD in $k$ for every $n>0$.  

For $n=0$, the above equation only implies that $\{\alpha_0(k)\}$ is RD in $k$.
We have the following difference equation:
\begin{equation*}
\alpha_n(k) = \beta_n(k+1)-\beta_n(k)\,.
\end{equation*}
This equation has a solution of the form
\begin{equation}\label{beta_alpha}
\beta_n(k) = -\sum_{r=k}^\infty\alpha_n(r)\,.
\end{equation}
It follows, since $\{\alpha_0(k)\}$ is RD in $k$, so is $\{\beta_0(k)\}$.  Thus $\{\beta_n(k)\}$ is RD for any $n$ and the formula \eqref{d_n_formula} extends by continuity to any $a\in A_S^\infty$.

We want to establish that $\{\beta_n(k)\}$ is RD in both $n$ and $k$. Since $d_n(U)\in\mathcal{K}^\infty$ we have that $\|d_n(U)\|_{M,N}$ are finite for all $M$ and $N$. So, for any $N$ and $j$ there exists a constant $C_{j,N}$ such that
\begin{equation*}
\|d_\K^j(d_n(U))(I+\K)^N\|\le C_{j,N}
\end{equation*}
On the other hand, consider the following calculation for $n\geq 0$:
\begin{equation*}
d_\K^j(d_n(U)) = d_\K^j(U^{n+1}\alpha_n(\K)) = (n+1)^jU^{n+1}\alpha_n(\K)
\end{equation*}
since $\alpha_n(\K)$ is diagonal.  Therefore, we have that
\begin{equation*}
(n+1)^j\|\alpha_n(\K)(I+\K)^N\|\le C_{j,N}\,.
\end{equation*}
However,
\begin{equation*}
(n+1)^{j}\|\alpha_n(\K)(I+\K)^N\| = (n+1)^{j}\underset{k}{\textrm{sup}}\left\{(1+k)^{N}|\alpha_n(k)|\right\}\,.
\end{equation*}
It follows that 
\begin{equation*}
(1+n)^j(1+k)^N|\alpha_n(k)|\le C_{j,N}
\end{equation*}
and thus $\{\alpha_n(k)\}$ is RD in both $n$ and $k$.  Consequently, by \eqref{beta_alpha}, $\{\beta_n(k)\}$ is RD in both $n$ and $k$. Therefore
\begin{equation*}
\begin{aligned}
d(a) &= \sum_{n\in\Z} d_n(a) = \sum_{n\ge0}[U^n\beta_n(\K),a] + \sum_{n<0}[\beta_n(\K)(U^*)^{-n},a] \\
&= \left[\sum_{n\ge0} U^n\beta_n(\K) + \sum_{n<0}\beta_n(\K)(U^*)^{-n},a\right]=[c,a]
\end{aligned}
\end{equation*}
where all the sums converge and $c\in\mathcal{K}^\infty$.  Thus $d$ is inner, completing the proof.
\end{proof}

To analyze general derivations  $d:A_S^\infty\to A_S^\infty$  we first notice the following.

\begin{prop}
Let $d:A_S^\infty\to A_S^\infty$ be a continuous derivation, then $d(\mathcal{K}^\infty)\subseteq\mathcal{K}^\infty$.
\end{prop}
\begin{proof}
Since $\mathcal{K}^\infty$ is generated by the system of units $\{P_{ks}\}$ and since $d$ is continuous we only need to verify that $d(P_{ks})$ is  in $\mathcal{K}^\infty$.
Since $P_{ks}=P_{kr}P_{rs}$, by the Leibniz rule we have that
\begin{equation*}
d(P_{ks})  = P_{kr}d(P_{rs}) + d(P_{kr})P_{rs}\,.
\end{equation*}
Since the right-hand side is clearly in $\mathcal{K}^\infty$, the claim follows.
\end{proof}

It follows from this proposition that any continuous derivation  $d:A_S^\infty\to A_S^\infty$ defines a continuous derivation on $B_S^\infty$, which is isomorphic to the factor algebra $A_S^\infty/\mathcal{K}^\infty$. We use this observation in the proof of the following main result of this section.

\begin{theo}
Let $d:A_S^\infty\to A_S^\infty$ be any continuous derivation.  Then there exist: a constant $\gamma$, $b\in B_S^\infty$ and $c\in\mathcal{K}^\infty$ such that:
\begin{equation*}
d = \gamma d_\K + [T(b)+c,\cdot].
\end{equation*}
\end{theo}
\begin{proof}
Let $d:A_S^\infty\to A_S^\infty$ be a continuous derivation and define a derivation $\delta:B_S^\infty\to B_S^\infty$ by
\begin{equation*}
\delta(a+\mathcal{K}^\infty) = d(a) + \mathcal{K}^\infty\,.
\end{equation*}
In other words, $\delta$ is the class of $d$ in the factor algebra $A_S^\infty/\mathcal{K}^\infty\cong B_S^\infty$. The continuity of $d$ implies the continuity of $\delta$.  But all continuous derivations $\delta:B_S^\infty\to B_S^\infty$ were classified in \cite{KMP2}.  Therefore, by that paper, there exists a constant $\gamma$ such that
\begin{equation*}
\delta = \gamma\delta_{\mathbb L} + \tilde{\delta}
\end{equation*}
where $\tilde{\delta}$ is inner.  Thus there exists a $b\in B_S^\infty$ such that $\tilde{\delta} = [b,\cdot]$. 

Next notice that $[T(b), \cdot]$ is an inner derivation on $A_S^\infty$ whose class in $B_S^\infty$ is precisely $[b,\cdot]$.  
Define a derivation $\tilde{d}:A_S^\infty\to A_S^\infty$ by
\begin{equation*}
\tilde{d} = d - cd_\K - [T(b),\cdot]\,.
\end{equation*}
Since the class of $d_\K$ in is $\delta_{\mathbb L}$, we have that $\tilde{d}:A_S^\infty\to\mathcal{K}^\infty$ and hence by Theorem \ref{cont_der_ran_com_inner}, $\tilde{d}=[c,\cdot]$ for some $c\in\mathcal{K}^\infty$. This concludes the proof.
\end{proof}

\section{K-Theory and $K$-Homology}
Since $\mathcal{K}^\infty, A_S^\infty, B_S^\infty$ are closed under the holomorphic functional calculus, each inclusion induces an isomorphism in $K$-Theory. Using this fact, along with the $6$-term exact sequence \cite{RLL} induced by the short exact sequence of smooth subalgebras, we compute the $K$-Theory of $A_S^\infty$. We then make use of the Universal Coefficient Theorem \cite{RS} to compute the $K$-Homology of $A_S$.  
\subsection{K Theory}
Recall the short exact sequence 
\begin{equation*}
0\longrightarrow \mathcal{K}^\infty \longrightarrow A^\infty_S \longrightarrow B^\infty_S \longrightarrow 0
\end{equation*}
of smooth subalgebras. This induces the following $6$-term exact sequence in $K$-Theory: 
\begin{equation*}
\begin{tikzcd}
 K_0(\mathcal{K}^\infty) \arrow{r}  & K_0(A^\infty_S) \arrow{r}{K_0(\tau )} & K_0(B^\infty_S) \arrow{d}{\exp}    \\
 K_1(B^\infty_S) \arrow{u}{\textrm{ind}} & K_1(A_S^\infty) \arrow{l}{K_1(\tau )} & K_1(\mathcal{K}^\infty) \arrow{l}
\end{tikzcd}
\end{equation*}
For details regarding the $K$-Theory of $B^\infty_S$, see \cite{KMP2}. Since the generating unitary $V$ in $B^\infty_S$ lifts to the partial isometry $U$, it follows that 
$$\textrm{ind}([V]_1) = [I - U^*U]_0 - [I - UU^*]_0 = -[P_{00}],$$ 
which generates $K_0(\mathcal{K}^\infty)$. Hence, the index map is an isomorphism. By exactness, it follows that $K_1(\tau )$ is the trivial map. Since $K_1(\mathcal{K}^\infty) = 0 $, by exactness $K_1(\tau )$ is also injective, and hence $K_1(A^\infty_S) = 0$. Since $\exp$ is trivial, by exactness $K_0(\tau )$ is surjective. But again, since $\textrm{ind}$ is an isomorphism, it follows that the map $K_0(\mathcal{K}^\infty) \to K_0(A^\infty_S)$ is trivial. Hence, $K_0(\tau )$ is injective as well. Using the computation done in \cite{KMP2}, it follows that we have: 
$$
K_0(A^\infty_S) \cong G_S\ \textrm{ where }\ G_S = \{ k/l \in \mathbb{Q} : k \in \mathbb{Z}, l | S \}. 
$$
Let us summarize the results in the following proposition. 
\begin{prop}
The $K$-Theory of $A_S$ is given by 
$$K_0(A_S) \cong G_S\ \textrm{ and }\ K_1(A_S) \cong 0.$$
\end{prop}
\subsection{$K$-Homology}
The Universal Coefficient Theorem of Rosenberg and Schochet \cite{RS} states that we have two exact sequences: 
\begin{equation*}
    \begin{tikzcd}
    0 \arrow{r} & \Ext^1_\Z(K_1(A_S), \Z) \arrow{r} & K^0(A_S) \arrow{r} & \Hom(K_0(A_S),\Z) \arrow{r} & 0,
    \end{tikzcd}
\end{equation*}
and
\begin{equation*}
    \begin{tikzcd}
    0 \arrow{r} & \Ext^1_\Z(K_0(A_S), \Z) \arrow{r} & K^1(A_S) \arrow{r} & \Hom(K_1(A_S),\Z) \arrow{r} & 0,
    \end{tikzcd}
\end{equation*}
where in the above, we have used the identification $KK^i(A_S, \mathbb{C}) = K^i(A_S)$. From the first sequence, it is clear that 
$$\Ext^1_\Z(K_1(A_S), \Z) \cong 0.$$
In \cite{KMP2} it was shown that 
$$\Hom(K_0(A_S),\Z) \cong 0.$$ 
Hence, we have $K^0(A_S) = 0$. From the second sequence, it is immediate that 
$$K^1(A_S) \cong \Ext^1_\Z(K_0(A_S), \Z) \cong K^1(B_S),$$ 
where the last isomorphism is derived in \cite{KMP2}.  This group was computed in \cite{KMP2} to be isomorphic to $(\Z / S \Z ) / \Z.$ This reference also contains an explicit description of the precise subgroup being modded out. In fact, this subgroup turns out to be the natural dense copy of $\Z \subseteq \Z / S \Z$. We summarize the above computations in the following proposition. 
\begin{prop}
The $K$-Homology of $A_S$ is given by 
$$K^0(A_S) \cong 0\ \textrm{ and }\ K^1(A_S) \cong (\Z / S \Z) / \Z.$$ 
\end{prop}

\end{document}